\documentclass[11pt]{article}

\usepackage{amssymb,amsmath,amsthm,amsfonts,fancybox,mathrsfs}
\usepackage[Symbol]{upgreek}

\theoremstyle{theorem}
\newtheorem{theorem}{Theorem}[section]
\newtheorem{prop}[theorem]{Proposition}
\newtheorem{lemma}[theorem]{Lemma}
\newtheorem{corollary}[theorem]{Corollary}

\theoremstyle{remark}
\newtheorem{remark}[theorem]{Remark}

\theoremstyle{definition}
\newtheorem{definition}[theorem]{Definition}
\newtheorem{example}[theorem]{Example}

\numberwithin{equation}{section}

\newcommand{\E}{\mathbf{E}}
\def\R{\mathbb{R}}
\def\HH{{\mathbb H}}
\def\LL{{\mathbb L}}

\newcommand{\Prob}[1]{\mathbf P \left\{#1 \right\}}

\newcommand{\one}{\mathbf{1}}
\newcommand{\thf}{\frac{1}{2}}
\newcommand{\eps}{\varepsilon}
\newcommand{\efrak}{\mathfrak{e}}
\newcommand{\abs}[1]{\lvert #1 \rvert}

\begin{document}
\bibliographystyle{plain}

\title{Stationarity of multivariate particle
  systems}
\author{\textsc{Ilya Molchanov and Kaspar Stucki}\\
  \normalsize
  Department of Mathematical Statistics and Actuarial Science,\\
  \normalsize
  University of Bern, Sidlerstrasse 5, CH-3012 Bern, Switzerland\\
  \normalsize E-mail: ilya@stat.unibe.ch,  kaspar@stucki.org}
\maketitle

\begin{abstract}
  \noindent A particle system is a family of i.i.d. stochastic 
  processes with values translated by Poisson points. We obtain 
  conditions that ensure the stationarity in time of the particle 
  system in $\R^d$ and in some cases provide a full characterisation 
  of the stationarity property. In particular, a full characterisation
  of stationary multivariate Brown--Resnick processes is given.
  
\vspace*{3mm}

\end{abstract}

\maketitle

\section{Introduction}
\label{sec:Intro}

A Poisson process in the Euclidean space $\R^d$ is stationary if its
intensity measure is proportional to the Lebesgue measure. More
general Poisson processes can be defined on richer spaces, e.g. the
space of functions or sets. While in these cases often there is no
analogue of the Lebesgue measure, invariance properties of the process
can be defined with respect to transformations that account for the
intrinsic structure of the relevant phase space, see
\cite[Ch.~3]{res87}.

One of most spectacular examples of this situation is due to Kabluchko
\cite{kab10}, who considered the following situation. Let $\Pi$ be a
Poisson point process on $\R$ and let $\{\xi_i,i\geq1\}$ be
i.i.d. copies of a real-valued stochastic process $\xi(t)$,
$t\in\R^m$. Define the family of functions $x_i+\xi_i(t)$, $t\in\R^m$,
for $x_i\in\Pi$, which (under appropriate integrability conditions on
the intensity of $\Pi$) becomes a point process on the space of
functions on $\R^m$.  For any $t\in\R^m$,
\begin{math}
  N(t)=\{x_i+\xi_i(t):\; i\geq1\}
\end{math}
is the Poisson point process on $\R$. Sometimes, the point process
$N(t)$ formed by the values of the translated function is stationary
in time even if $\xi$ is not stationary.  It is important to
distinguish this concept from the stationarity on $\R$, where the
points lie.

Kabluchko \cite{kab10} characterised the cases when a real-valued
Gaussian process $\xi$ gives rise to a stationary point system $N(t)$
called a stationary \emph{Gaussian system} assuming that the intensity
measure $\Lambda$ of $\Pi$ satisfies $\int_\R e^{-\eps
  x^2}\Lambda(dx)<\infty$ for all $\eps>0$. All stationary Gaussian
systems are given by the following three classes.
\begin{itemize}
\item[(i)] $\Lambda$ is an arbitrary measure on $\R$ and $\xi$ is a
  stationary Gaussian process.
\item[(ii)] $\Lambda$ is proportional to the Lebesgue measure on $\R$
  and $\xi(t)=W(t)+b(t)+c$, where $W$ is a centred Gaussian process
  with stationary increments, $b$ is an \emph{additive function},
  i.e. $b(t+s)=b(t)+b(s)$ for all $t$ and $s$, and $c\in\R$ is a
  constant.
\item[(iii)] The density of $\Lambda$ is proportional to $e^{-\lambda
    x}$, $x\in\R$, with $\lambda\neq0$, and $\xi(t)=W(t) -\lambda
  \sigma^2(t)/2+c$, where $W$ is a centred Gaussian process with
  stationary increments and variance $\sigma^2(t)$, and $c\in\R$ is a
  constant.
\end{itemize}

The aim of this paper is to provide a partial generalisation of the
above result for the case when $\xi$ takes values in a
higher-dimensional Euclidean space, which is also mentioned in
\cite{kab10} as an interesting open problem. In some
cases, notably for multivariate Brown--Resnick processes, our
characterisation is complete. The current work also
yields alternative proofs of some results from \cite{kab10}.

\section{Multivariate particle systems}
\label{sec:mult-point-syst}

Let $\{\xi_i,i\geq1\}$ be i.i.d. copies of a $\R^d$-valued stochastic
process $\xi(t)$, $t\in\R$. All subsequent results can be easily
generalised and remain valid for processes $\xi$ with argument $t$
from a higher-dimensional Euclidean space.

Furthermore, let $\Pi=\{x_i, \; i\geq1\}$ be a Poisson point process
in $\R^d$ independent of the $\xi_1,\xi_2,\ldots$. We call the process
\begin{displaymath}
  N(t)=\{x_i+\xi_i(t), i \geq 1\}\,,\qquad t\in \R\,,
\end{displaymath}
a \emph{particle system}, so that a particle system is a stochastic
process with values in the space of point configurations (or counting
measures). Since the distribution of $\Pi$ is completely determined by
its intensity measure $\Lambda$, we say that the particle system
$(\Lambda,\xi)$ is generated by measure $\Lambda$ and the process $\xi$. If the process $\xi$ is Gaussian, we call $(\Lambda,\xi)$ a \emph{Gaussian system}.

By the finite-dimensional distributions of $N$ we mean the
distribution of the point process in $\R^{dn}$ given by
\begin{displaymath}
  N(t_1,\dots,t_n)
  =\{(x_i+\xi_i(t_1),\dots,x_i+\xi_i(t_n)), i\geq1\}\,,\quad
  t_1,\dots,t_n\in\R\,.
\end{displaymath}

Denote by $P_{t_1,\dots,t_n}$ the finite-dimensional distributions of
$\xi$, in particular $P_t$ is the distribution of $\xi(t)$.  From now
on we always assume that the convolution $\Lambda*P_t$ is a locally
finite measure for all $t \in \R$.  The following result is easy to
obtain using the probability generating functional of the Poisson
process, see \cite[Ex.~9.4(c)]{dal:ver08}.

\begin{prop}
  If $\Lambda*P_t$ is a locally finite measure for all $t \in \R$,
  then, for all $t_1,\dots,t_n\in\R$, $N(t_1,\dots,t_n)$ is a Poisson
  point process in $\R^{dn}$ with locally finite intensity measure
  \begin{equation}
    \label{eq:Lambda}
    \Lambda_{t_1,\dots,t_n}(A)
    =\int_{\R^d} P_{t_1,\dots,t_n}(A-x)\Lambda(dx)
  \end{equation}
  for all Borel $A \subset\R^{dn}$, where $A-x$ is $A$ translated by
  $(x,\dots,x)$ composed of $n$ copies of $x\in\R^d$.
\end{prop}

The main question addressed in this paper is to characterise all pairs
$(\Lambda,\xi)$, such that the corresponding particle system $N$ is
stationary. By stationarity we mean that for all $s,t_1,\dots,t_n\in
\R$ the distributions of $N(t_1,\dots,t_n)$ and $N(t_1+s,\dots,t_n+s)$
coincide.  Since the distribution of a Poisson point process is
determined by its intensity measure, we immediately obtain the
following result.

\begin{prop}
  \label{prop:stat}
  The particle system generated by $\Lambda$ and $\xi$ is
  stationary if and only if
  \begin{equation}
    \label{eq:stat}
    \Lambda_{t_1,\dots,t_n}=\Lambda_{t_1+s,\dots,t_n+s}
  \end{equation}
  for all $s,t_1,\dots,t_n \in \R$.
\end{prop}

\section{Convolution equations}
\label{sec:conv-equat}

The stationarity condition (\ref{eq:stat}) is in fact a system of
convolution equations of the form
\begin{equation}
  \label{eq:lam}
  P_{t_1,\dots,t_n}*\tilde{\Lambda}
  =P_{t_1+s,\dots,t_n+s}*\tilde{\Lambda}\,,
\end{equation}
where $\tilde{\Lambda}$ is the measure obtained by uplifting $\Lambda$
to the diagonal in $\R^{dn}$. In general notation, these equations are
of the type
\begin{equation}
  \label{eq:Deny2}
  \sigma_1*\mu=\sigma_2*\mu \,,
\end{equation}
where $\sigma_1$ and $\sigma_2$ are probability measures and $\mu$ is
an unknown locally finite measure on $\R^d$.  If $\sigma_2$ can be
decomposed as $\sigma_2=\sigma_1*\sigma$ (or if
$\sigma_1=\sigma_2*\sigma$), then (\ref{eq:Deny2}) simplifies to
\begin{equation}
  \label{eq:Deny1}
  \mu=\mu*\sigma
\end{equation}
for another measure $\mu$.  This convolution equation was solved by
D\'eny \cite{den59}. Namely, if the support of $\sigma$ is the whole
$\R^d$, then all solutions of \eqref{eq:Deny1} are mixtures of
exponential measures. i.e.
\begin{equation}
  \label{eq:mixture}
  \mu=\int_E \efrak_\lambda Q(d\lambda) \,,
\end{equation}
where $\efrak_\lambda$ is the measure on $\R^d$ with density
$e^{-\langle \lambda,x\rangle}$, $x\in\R^d$, and $Q$ is a measure on
the set $E=E_\sigma$ with
\begin{equation}
  \label{eq:Esigma}
  E_\sigma=\Big\{\lambda \in \R^d:
  \int_{\R^d} e^{\langle \lambda,x\rangle}\sigma(dx)=1\Big\}\,.
\end{equation}
In particular, if $\xi$ is a real-valued Gaussian process with
non-constant variance $\sigma^2(t)$, $t\in\R$, then there exist
$t_1,t_2\in\R$ such that $\sigma^2(t_2)>\sigma^2(t_1)$, so that the
first convolution equation $P_{t_1}*\Lambda=P_{t_2}*\Lambda$ can be
reduced to the D\'eny convolution equation \eqref{eq:Deny1} for
$\sigma$ being the normal law with the variance
$\sigma^2(t_2)-\sigma^2(t_1)$.  Hence $\Lambda*P_t$ is a mixture of
exponential measures, which is the crucial argument in the
characterisation of stationary Gaussian systems in \cite{kab10}.

In the multivariate case it is usually not possible to reduce the two-sided convolution equation to a one-sided equation,
since the difference of two covariance matrices may be neither positive nor negative definite. In the
spirit of \eqref{eq:Esigma}, define
\begin{equation}
  \label{eq:E}
  E_{\sigma_1\sigma_2}=\Big\{\lambda \in \R^d:
  \int_{\R^d}e^{\langle \lambda,x\rangle}\sigma_1(dx)
  =\int_{\R^d}e^{\langle \lambda,x\rangle}\sigma_2(dx)\Big\}\,.
\end{equation}
While each measure $\mu$ given by (\ref{eq:mixture}) with
$E=E_{\sigma_1\sigma_2}$ satisfies the convolution equation
(\ref{eq:Deny2}), there exist solutions of \eqref{eq:Deny2} not in the
form (\ref{eq:mixture}).

\begin{example}
  \label{ex:counterexample}
  Let $\sigma_1$ and $\sigma_2$ be bivariate centred normal
  distributions with covariance matrices 
  \begin{displaymath}
    \Sigma_1=\begin{pmatrix}1+c_1^2 & 0\\ 0 & 1 \end{pmatrix}
    \quad\text{and}\quad
    \Sigma_2=\begin{pmatrix}1 & 0\\ 0 & 1+c_2^2 \end{pmatrix} 
  \end{displaymath}
  for some constants $c_1,c_2>0$. Let $g:\R\to\R$ be a function such
  that $\int_\R g(x+y)e^{-y^2/2}dy$ is finite for all $x\in\R$. Then
  measure $\mu_g$ with the density $g(c_1^{-1}x_1+c_2^{-1}x_2)$
  satisfies (\ref{eq:Deny2}).  Indeed, substitution $z=c_1^{-1}c_2y$
  yields that
  \begin{displaymath}
    \frac{1}{c_1}\int_\R g\Big(\frac{x_1}{c_1}+\frac{x_2}{c_2}
    -\frac{y}{c_1}\Big)e^{-\frac{1}{2}\left(\frac{y}{c_1}\right)^2}dy
    =\frac{1}{c_2}\int_\R g\Big(\frac{x_1}{c_1}+\frac{x_2}{c_2}
    -\frac{z}{c_2}\Big)e^{-\frac{1}{2}\left(\frac{z}{c_2}\right)^2}dz\,.
  \end{displaymath}
  It remains to note that the two sides of this equality are up to the
  same constant the densities of the convolution of $\mu_g$ and the
  centred normal distributions with covariance matrices
  \begin{displaymath}
    \begin{pmatrix} c_1^2 & 0\\ 0 & 0 \end{pmatrix}
    \quad\text{and}\quad 
    \begin{pmatrix} 0 & 0\\ 0 & c_2^2 \end{pmatrix} 
  \end{displaymath}
  Note that $E_{\sigma_1\sigma_2}=\{(\lambda_1,\lambda_2):\;
  c_1|\lambda_1|=c_2|\lambda_2|\}$.  For instance, if $g(x)=x^2$ then
  there exists no measure $Q$ such that 
  $\mu_g=\int_{E_{\sigma_1\sigma_2}} \efrak_\lambda Q(d\lambda)$.
\end{example}

Unfortunately there is no general result describing solutions of
(\ref{eq:Deny2}). The two-sided convolution equation can be written as
$\mu*\nu=0$ for a signed measure $\nu$ with finite total variation. If
$\nu$ has bounded support, then the density of $\mu$ solving this
equation is called a mean periodic function. Typical examples of mean
periodic functions are exponential polynomials, i.e. sums of products
of polynomials and exponential functions. While exponential
polynomials are dense in the family of mean periodic functions on the
line \cite{ehr55}, this is unknown for higher-dimensional spaces. The
situation with $\nu$ having unbounded support (e.g. corresponding to
the difference of two Gaussian measures on $\R^d$) is even less
explored.

\section{Multivariate stationarity}
\label{sec:Mstat}

In this section we characterise the stationarity conditions for some
(but still rather general) families of intensity measures $\Lambda$.

\subsection{Exponential measures.}

Consider candidates for the solutions of \eqref{eq:lam} of the form
$\Lambda=\efrak_\lambda$ for $\lambda\in\R^d$. It is easy to see that
necessarily $\lambda\in E_{P_tP_s}$ (see \eqref{eq:E}) for any
$t,s\in\R$. The convolution $\efrak_\lambda*P_t$ is locally finite if and
only if
\begin{equation}
  \label{eq:intexp}
  \E e^{\langle \lambda,\xi(t)\rangle} < \infty
  \quad \text{for all } t \in \R \,.
\end{equation}
Then the characteristic function with a complex argument in its first
coordinate
\begin{multline*}
  \varphi_{t_1,\dots,t_n}(u_1-\imath\lambda,u_2,\dots,u_n)\\
  =\E \exp\{\imath(\langle (u_1-\imath\lambda),\xi(t_1)\rangle+
  \langle u_2, \xi(t_2)\rangle +\dots
  + \langle u_n, \xi(t_n)\rangle)\} 
\end{multline*}  
exists for all $u_1,\dots,u_n \in \R^d$, where $\imath$ is the
imaginary unit.

\begin{theorem}
  \label{thm:Laplace1e}
  Assume that (\ref{eq:intexp}) holds. The particle system generated
  by $\efrak_\lambda$ and $\xi$ is stationary if and only if
  \begin{equation}
    \label{eq:Fourier}
    \varphi_{t_1,\dots,t_n}(u_1-\imath\lambda,u_2,\dots,u_n)
    =\varphi_{t_1+s,\dots,t_n+s}(u_1-\imath\lambda,u_2,\dots,u_n)
  \end{equation}
  for all $n\geq1$, $s,t_1,\dots,t_n \in \R$ and
  $u_1,\dots,u_n\in\R^d$ satisfying $\sum_{i=1}^n u_i = 0$.
\end{theorem}
\begin{proof}
  The proof follows the idea of \cite[Prop.~6]{kab:sch:haan09}.
  Let $A$ be a bounded Borel set in $\R^{dn}$. Then
  \begin{align}
    \Lambda_{t_1,\dots,t_n}(A)
    &= \E \int_{\R^d}
    \mathbf{1}_{(\xi(t_1)+x,\dots,\xi(t_n)+x) \in A}\;
    e^{-\langle\lambda,x\rangle}dx\nonumber \\
    &= \int_{\R^d} \mu_{t_1,\dots,t_n}(A-z)
    e^{-\langle\lambda,z\rangle}dz\,,
    \label{eq:mu-disint}
  \end{align}
  where $\mu$ is a measure on $\R^{dn}$ given by
  \begin{equation}
    \label{eq:mut}
    \mu_{t_1,\dots,t_n}(A)
    = \E \Big[\mathbf{1}_{(0,\xi(t_2)-\xi(t_1),\dots,\xi(t_n)-\xi(t_1))
    \in A}\; e^{\langle\lambda,\xi(t_1)\rangle}\Big]\,.
  \end{equation}
  Since $\mu_{t_1,\dots,t_n}$ is supported by the subspace
  $\{(x_1,\dots,x_n)\in\R^{dn}:\; x_1=0\}$, the decomposition
  $\Lambda_{t_1,\dots,t_n}(A)=\int\mu_{t_1,\dots,t_n}(A-z)e^{-\langle\lambda,
    z\rangle}dz$ is unique, e.g. see
  \cite[Th.~15.3.3]{kalle83}. Finally, note that the Fourier transform
  of $\mu_{t_1,\dots,t_n}$ is given by
  \begin{displaymath}
    \hat{\mu}_{t_1,\dots,t_n}(u_1,\dots,u_n)
    = \varphi_{t_1,\dots,t_n}\left(-\imath\lambda
      -\sum_{i=2}^n u_i,u_2,\dots,u_n \right) \,.
  \end{displaymath}
\end{proof}

A similar proof with the Laplace transform instead of the Fourier
transform yields the following result.

\begin{prop}
  \label{cor:FL}
  Assume that the Laplace transform 
  \begin{equation}
    \psi_{t_1,\dots,t_n}(u_1,\dots,u_n)
    =\E \exp\{\langle u_1,\xi(t_1)\rangle+\dots
    + \langle u_n, \xi(t_n)\rangle\} 
  \end{equation}  
  exists for all $u_1,\dots,u_n \in \R^d$ such that $\sum_{i=1}^n
  u_i=\lambda$. Then the particle system $(\efrak_\lambda,\xi)$ is
  stationary if and only if
  \begin{math}
    \psi_{t_1,\dots,t_n}(u_1,\dots,u_n)
    =\psi_{t_1+s,\dots,t_n+s}(u_1,\dots,u_n)
  \end{math}
  for all $n\geq1$, $s,t_1,\dots,t_n\in\R$ and $u_1,\dots,u_n\in\R^d$
  satisfying $\sum_{i=1}^n u_i=\lambda$.
\end{prop}

For Gaussian processes we can give a more precise statement.  Denote
by $\Sigma(t_1,t_2)$ the covariance matrix of $\xi(t_1)$ and
$\xi(t_2)$, in particular $\Sigma(t,t)$ is the covariance matrix of
$\xi(t)$. It is important to note that, unlike in the univariate case,
$\Sigma(t_1,t_2)$ may differ from $\Sigma(t_2,t_1)$, namely
$\Sigma(t_2,t_1)=\Sigma(t_1,t_2)^\top$.

\begin{example}
  \label{ex:wp-shifted}
  Let $\xi^1(t)=W(t)$ and let $\xi^2(t)=W(t+h)$ for some fixed $h>0$,
  where $W$ is the Wiener process. Then $\E\xi^1(t_1)\xi^2(t_2)$ is
  not necessarily equal to $\E\xi^1(t_2)\xi^2(t_1)$, so that
  $\Sigma(t_1,t_2)$ is not necessarily symmetric. 
\end{example}

The covariance matrix (variogram) of $\xi(t_2)-\xi(t_1)$ is given by
\begin{displaymath}
  \Gamma(t_1,t_2)=\Sigma(t_2,t_2)-\Sigma(t_1,t_2)
  -\Sigma(t_2,t_1)+\Sigma(t_1,t_1)\,.
\end{displaymath}
We say that multivariate Gaussian process $\xi$ has \emph{wide sense
  stationary increments} if and only if $\Gamma (t_1,t_2)$ depends
only on the difference $t_1-t_2$. In the univariate case, this
property is equivalent to the fact that $\xi(t+s)-\xi(t)$, $t\in\R$,
is stationary for each $s\in\R$, see \cite[Lemma~1]{kab10}, while in
the multivariate case this is not so.

\begin{theorem}
  \label{thm:Gauss1e}
  The measure $\efrak_\lambda$ and a Gaussian process $\xi$ generate a
  stationary particle system if and only if
  \begin{equation}
    \label{eq:Gauss1e}
    \xi(t) = W(t) - \frac{1}{2}\Sigma(t,t)\lambda +b(t) +c \,, \quad
    t\in\R\,,
  \end{equation}
  where $W$ is a centred Gaussian process with wide sense stationary
  increments and variance $\Sigma(t,t)$, $c\in\R^d$ is deterministic,
  and $b:\R\to\R^d$ is a function orthogonal to $\lambda$ such that
  \begin{equation}
    \label{eq:b}
    b(t_2)-b(t_1)+\frac{1}{2}(\Sigma(t_2,t_1)-\Sigma(t_1,t_2))\lambda
  \end{equation}
  depends only on the difference $t_2-t_1$.
\end{theorem}

\begin{remark}
  \label{remark:lambda0}
  If $\lambda=0$, condition (\ref{eq:b}) implies that $b(t)-b(0)$ is
  an additive function, see \cite[Lemma~2]{kab10}. This is also the
  case if $\Sigma(t_1,t_2)$ is symmetric for all $t_1$ and $t_2$,
  e.g. in the univariate case where the orthogonality of $b$ and
  $\lambda$ implies that $b$ vanishes if $\lambda\neq0$. 
\end{remark}

We use the following lemma, that is is easy to prove by  direct
computation. 

\begin{lemma}
  \label{lemma:lt}
  Consider all Gaussian vectors in the Euclidean space $\R^n$ whose
  Laplace transform $\psi(u)$ is given for all $u$ from $\LL+a$, where
  $\LL$ is a linear subspace of $\R^n$ and $a\in\R^n$. Then all these
  vectors share the same values of $A^\top\Sigma A$, $A^\top(m+\Sigma
  a)$ and $\langle m,a\rangle+\thf\langle a,\Sigma a\rangle$, where
  $m$ and $\Sigma$ are the mean and covariance matrix of the
  corresponding vector and $A$ denotes any projection of $\R^n$ onto
  $\LL$.
\end{lemma}

\begin{proof}[Proof of Theorem~\ref{thm:Gauss1e}]
  The sufficiency follows by explicit writing of the Laplace
  transform of $(\xi(t_1),\dots,\xi(t_n))$.  For the necessity, let
  $n=2$ and apply Lemma~\ref{lemma:lt} and Proposition~\ref{cor:FL}
  with $\LL=\{(u_1,u_2)\in\R^{2d}:\; u_1+u_2=0\}$ and
  $a=(\lambda,0)\in\R^{2d}$. Define block matrices
  \begin{displaymath}
    A=
    \begin{pmatrix}
      I & -I\\
      -I & I
    \end{pmatrix}\,, \qquad
    \Sigma=
    \begin{pmatrix}
      \Sigma(t_1,t_1) & \Sigma(t_1,t_2)\\ \\
      \Sigma(t_2,t_1) & \Sigma(t_2,t_2)
    \end{pmatrix}\,,
  \end{displaymath}
  where $I$ is the $d$-dimensional unit matrix. Note that $A$ defines
  a projection on $\LL$ and $\Sigma$ is the covariance of
  $(\xi(t_1),\xi(t_2))$. Then all elements of $A^\top\Sigma A$ are
  proportional to $\Gamma(t_1,t_2)$, meaning that
  $W(t)=\xi(t)-\E\xi(t)$, $t\in\R$, has wide sense stationary
  increments.
  
  Define $m(t)=\E\xi(t)$.  Calculating $A^\top(m+\Sigma a)$ with
  $m=(m(t_1),m(t_2))$ it is easy to see that
  \begin{equation}
    \label{eq:mu-second}
    m(t_2)-m(t_1)+ (\Sigma(t_1,t_1)-\Sigma(t_2,t_1))\lambda
  \end{equation}
  is invariant after $(t_1,t_2)$ is replaced by
  $(t_1+s,t_2+s)$. Denoting 
  \begin{displaymath}
    b(t)=m(t)+\thf\Sigma(t,t)\lambda
  \end{displaymath}   
  and using the fact that $\Gamma(t_1,t_2)=\Gamma(t_1+s,t_2+s)$, we
  arrive at (\ref{eq:b}).  Furthermore,
  \begin{displaymath}
    \langle m,a\rangle+\thf \langle a,\Sigma
    a\rangle=\langle m(t_1),\lambda\rangle+\thf
    \langle\lambda,\Sigma(t_1,t_1)\lambda\rangle
    =\langle b(t),\lambda\rangle 
  \end{displaymath}
  does not depend on $t_1$, so that $\langle b(t),\lambda\rangle$ is
  constant. Finally, set $c=b(0)$ and replace $b(t)$ by $b(t)-b(0)$.
\end{proof}

\subsection{Mixtures of exponential measures.}

Consider particle systems generated by Poisson processes with intensity measures given by mixtures of $\efrak_\lambda$ for $\lambda\in E\subset\R^d$.

\begin{theorem}
  \label{lemma:mixture}
  Assume that $\xi$ is a stochastic process such that
  \eqref{eq:intexp} holds for all $\lambda$ from an open neighbourhood
  $U$ of $E\subset \R^d$ and the measure
  \begin{displaymath}
    \Lambda=\int_E\efrak_\lambda Q(d\lambda)
  \end{displaymath}  
  is locally finite, where $Q$ is a measure supported by $E$. Then the
  particle system generated by $\Lambda$ and $\xi$ is stationary if
  and only if, for all $\lambda \in E$, the system
  $(\efrak_\lambda,\xi)$ is stationary.
\end{theorem}
\begin{proof}
  We only need to prove the necessity. For $v \in \R^d$ define
  \begin{align*}
    E_1&=\{\lambda \in E:\, \langle\lambda,v\rangle <1 \} \,,\\
    E_2&=\{\lambda \in E:\, \langle\lambda,v\rangle \geq1 \} \,.
  \end{align*}
  Let $\Lambda_i=\int_{E_i} \efrak_\lambda Q(d\lambda)$,
  $i=1,2$. Without loss of generality assume that neither $\Lambda_1$
  nor $\Lambda_2$ is the zero measure.  Let $A$ be a bounded Borel
  set. Since $\Lambda$ satisfies (\ref{eq:stat}),
  \begin{equation}
    \label{eq:convcontra}
    \Lambda_1*(P_{t_1,\dots,t_n}-P_{t_1+s,\dots,t_n+s})(A)=
    \Lambda_2*(P_{t_1+s,\dots,t_n+s}-P_{t_1,\dots,t_n})(A) \,.
  \end{equation}
  Assume that (\ref{eq:convcontra}) is positive. Since $Q(E_1)>0$,
  \begin{align*}
    \Lambda_1*(P_{t_1,\dots,t_n}-P_{t_1+s,\dots,t_n+s})(A+v) &>
    e^{-1}\Lambda_1*(P_{t_1,\dots,t_n}-P_{t_1+s,\dots,t_n+s})(A) \,,\\
    \Lambda_2*(P_{t_1+s,\dots,t_n+s}-P_{t_1,\dots,t_n})(A+v) &\leq
    e^{-1}\Lambda_2*(P_{t_1+s,\dots,t_n+s}-P_{t_1,\dots,t_n})(A) \,.
  \end{align*}
  In view of (\ref{eq:convcontra}),
  \begin{displaymath}
    \Lambda_1*(P_{t_1,\dots,t_n}-P_{t_1+s,\dots,t_n+s})(A+v) >
    \Lambda_2*(P_{t_1+s,\dots,t_n+s}-P_{t_1,\dots,t_n})(A+v) \,.
  \end{displaymath}
  Rearranging the terms yields that
  \begin{displaymath}
    (\Lambda_1+\Lambda_2)*P_{t_1,\dots,t_n}(A+v)
    > (\Lambda_1+\Lambda_2)*P_{t_1+s,\dots,t_n+s}(A+v) \,,
  \end{displaymath}
  which contradicts that $\Lambda=\Lambda_1+\Lambda_2$ satisfies
  (\ref{eq:stat}).  A similar argument excludes the negativity of
  (\ref{eq:convcontra}), and therefore $\Lambda_1$ and $\Lambda_2$
  satisfy (\ref{eq:stat}) for all bounded Borel $A$.

  Consider any $\lambda_0\in E$. By cutting $E$ with hyperplanes, it
  is possible to construct a sequence of relatively compact sets
  $E_k\subset E$, $k\geq1$, such that $E_k\downarrow\{\lambda_0\}$,
  the closure of $E_1$ is a subset of $U$ and
  $\Lambda_k=\int_{E_k}\efrak_\lambda Q(d\lambda)$ satisfies
  (\ref{eq:stat}) for all $k$.  Since (\ref{eq:stat}) is scale
  invariant, it also holds for
  $\tilde{\Lambda}_k=\int_{E_k}\efrak_\lambda \tilde{Q}_k(d\lambda)$
  with $\tilde{Q}_k(\cdot)=Q(\cdot)/Q(E_k)$.  For all $k$,
  \begin{equation}
    \label{eq:sandwich}
    \inf_{\lambda \in E_k} e^{-\langle \lambda,x \rangle} \leq
    \int_{E_k} e^{-\langle\lambda,x\rangle} \tilde{Q}_k(d\lambda) \leq
    \sup_{\lambda \in E_k} e^{-\langle \lambda,x \rangle}  \,,
  \end{equation}
  since the both sides of (\ref{eq:sandwich}) converge to
  $e^{-\langle\lambda_0,x\rangle}$, $\tilde{\Lambda}_k(A)\to
  \efrak_{\lambda_0}(A)$ for all measurable $A$.

  It remains to show that the limiting measure satisfies
  (\ref{eq:stat}). By \eqref{eq:mu-disint},
  \begin{equation}
    \label{eq:muq}
    \tilde{\Lambda}_k*P_{t_1,\dots,t_n}(A)=\int_{\R^d}\int_{E_k} 
    \mu_{t_1,\dots,t_n}(A-x)e^{-\langle \lambda,x\rangle}
    \tilde{Q}_k(d\lambda) \, dx \,,
  \end{equation}
  where $\mu_{t_1,\dots,t_n}(A)$ is defined in \eqref{eq:mut}.  

  Let $A_1$ be the set of $x \in \R^d$ such that $(x,y) \in A$ for
  some $y\in\R^{d(n-1)}$. Since $\tilde{Q}_k(E_k)=1$ and
  $\mu_{t_1,\dots,t_n}(A-x) \le \mathbf{1}_{A_1}(x) \E e^{\langle
    \lambda,\xi(t_1) \rangle}$,
  \begin{displaymath}
    \int_{E_k} 
    \mu_{t_1,\dots,t_n}(A-x)e^{-\langle \lambda,x\rangle}
    \tilde{Q}_k(d\lambda)\le c\mathbf{1}_{A_1}(x)
    e^{-\langle \lambda,x\rangle}\,,
  \end{displaymath}
  where $c$ is the supremum of $\E e^{\langle
    \lambda,\xi(t_1)\rangle}$ for $\lambda$ from the closure of
  $E_1$. This supremum is finite, since $\E e^{\langle \lambda, \xi(t)
    \rangle}$ is analytic, hence continuous, in its domain $U$.  Since
  $A_1$ is bounded,
  \begin{displaymath}
    \int_{\R^d} c\mathbf{1}_{A_1}(x)
    e^{-\langle \lambda,x\rangle} dx <\infty
  \end{displaymath}
  and the Lebesgue dominated convergence theorem yields
  \begin{align*}
    \lim_{k \to \infty}\tilde{\Lambda}_k*P_{t_1,\dots,t_n}(A)
    =&\int_{\R^d}\lim_{k \to \infty}\int_{E_k} 
    \mu_{t_1,\dots,t_n}(A-x)e^{-\langle \lambda,x\rangle}
    \tilde{Q}_k(d\lambda) \, dx \\
    =&\int_{\R^d} \mu_{t_1,\dots,t_n}(A-x)e^{-\langle
      \lambda_0,x\rangle}dx
    =\efrak_{\lambda_0}*P_{t_1,\dots,t_n}(A) \,,
  \end{align*}
  where the second equality follows by a similar argument as
  (\ref{eq:sandwich}).
\end{proof}

Thus, if the particle system $(\Lambda,\xi)$ is
stationary and the conditions of Theorem~\ref{lemma:mixture} are
satisfied, then the support of the measure $Q$ is contained in
$E_{P_{t_1}P_{t_2}}$ (see \eqref{eq:E}) for all $t_1,t_2\in \R$.

\begin{prop}
  \label{prop:twoexp}
  Let $\lambda_1, \lambda_2 \in \R^d$ with $\lambda_1 \neq
  \lambda_2$. If the Gaussian systems $(\efrak_{\lambda_1},\xi)$ and
  $(\efrak_{\lambda_2},\xi)$ are stationary, then the
  one-dimensional stochastic process $\langle
  \xi-\E\xi,\lambda_2-\lambda_1\rangle$ is stationary.
\end{prop}
\begin{proof}
  Writing \eqref{eq:Gauss1e} for $\lambda_i$, $i=1,2$, we arrive at
  \begin{displaymath}
    W(t) - \frac{1}{2}\Sigma(t,t)\lambda_1 +b_1(t)+c_1
    =W(t) - \frac{1}{2}\Sigma(t,t)\lambda_2 +b_2(t)+c_2\,.
  \end{displaymath}
  Since $\Sigma(t_2,t_1)-\Sigma(t_1,t_2)$ is a skew symmetric matrix,
  (\ref{eq:b}) implies that
  \begin{equation}
    \label{eq:twolambda}
    \langle \lambda_2,b_1(t_2)-b_1(t_1) \rangle + 
    \langle \lambda_1,b_2(t_2)-b_2(t_1) \rangle
  \end{equation}   
  is invariant after $(t_1,t_2)$ is replaced by
  $(t_1+s,t_2+s)$. Denote shortly
  $\Delta\lambda=\lambda_1-\lambda_2$. Rewriting (\ref{eq:twolambda})
  yields
  \begin{multline*}
    \label{eq:deltaS}
    \langle \Delta\lambda,\Sigma(t_1,t_1)\Delta \lambda\rangle -
    \langle \Delta\lambda,\Sigma(t_2,t_2)\Delta \lambda\rangle\\ 
    =\langle \Delta\lambda,\Sigma(t_1+s,t_1+s)\Delta \lambda\rangle -
    \langle \Delta\lambda,\Sigma(t_2+s,t_2+s)\Delta \lambda\rangle \,.
  \end{multline*}
  By \cite[Lemma~2]{kab10}, the function $\langle
  \Delta\lambda,\Sigma(t,t)\Delta \lambda\rangle$ is an additive
  function plus a constant. In view of the positive definiteness of
  $\Sigma(t,t)$, we conclude that $\langle
  \Delta\lambda,\Sigma(t,t)\Delta \lambda\rangle$ is constant for all
  $t$. The statement follows from the fact that a univariate Gaussian
  process with stationary increments and constant variance is itself
  stationary.
\end{proof}

The following result characterises stationary particle systems in case
the two-sided D\'eny equation reduces to the one-sided one.

\begin{corollary}
  \label{cor:decomp}
  Assume that $(\Lambda,\xi)$ is a stationary particle system, where
  $\xi$ is a Gaussian process such that $P_{t_1}=P_{t_2}*\sigma$ for a
  Gaussian measure $\sigma$ and some $t_1\neq t_2$. If no linear
  combination of the components of $\xi-\E\xi$ is stationary, then
  $\Lambda=c\efrak_\lambda$ for some $c>0$ and $\xi$ is given by
  \eqref{eq:Gauss1e}.
\end{corollary}
\begin{proof}
  The D\'eny theorem implies that $\Lambda$ is a mixture of
  exponential measures, so the result follows from
  Theorem~\ref{lemma:mixture} and Proposition~\ref{prop:twoexp}.
\end{proof}

In particular, Corollary~\ref{cor:decomp} applies if $\xi(t)$ is
a.s. deterministic for at least one $t$, for instance if
$\xi(0)=0$. Furthermore, it yields the result of \cite{kab10} for
non-stationary univariate process $\xi$.

\begin{example}
  \label{ex:two-lambda}
  Let $\xi^1(t)=\xi^2(t)=W(t)-a\abs{t}/2$, where $W$ is the two-sided
  Brownian motion and $a\in\R$. Then $\Lambda=\int_\R
  \efrak_{(a+\lambda,-\lambda)}Q(d\lambda)$ for a measure $Q$ on $\R$
  satisfying the integrability condition and
  $\xi=(\xi^1,\xi^2)$ generate a stationary particle system.
\end{example}

\subsection{Measures with exponential polynomial densities.}

Assume that $\Lambda$ has the density
\begin{displaymath}
  p(x)e^{-\langle
  \lambda,x\rangle}=\sum_{\abs{\alpha}\le k}c_\alpha x^\alpha e^{\langle
  -\lambda,x\rangle}\,,
\end{displaymath}
where $p(x)$ is a non-negative polynomial of degree $k$. We use the
multi-index notation, i.e. $\alpha=(\alpha^1,\dots,\alpha^d)$,
$\abs{\alpha}=\alpha^1+\dots+\alpha^d$ and
$x^\alpha=(x^1)^{\alpha^1}\cdots (x^d)^{\alpha^d}$. Note that one can
also consider solutions of convolution equations with not necessarily
non-negative polynomials, which however do not admit an interpretation
as intensities of point processes. Nonetheless, even then we speak
about stationary particle systems.

\begin{theorem}
  \label{thm:polynomial}
  If the particle system $(p(x)e^{-\langle\lambda,x\rangle}, \xi)$ for
  a polynomial $p$ is stationary, then the particle system
  $(q(x)e^{-\langle\lambda,x\rangle}, \xi)$ is stationary for each
  polynomial $q$ obtained as a partial derivative of $p$.
\end{theorem}
\begin{proof}
  For each $n$, bounded Borel set $A$ in $\R^{dn}$ and $x\in\R^d$,
  \begin{align*}
    \Lambda_{t_1,\dots,t_n}(A+x) 
    & = \int_{\R^d} P_{t_1,\dots,t_n}(A+x-z)p(z)e^{-\langle
      \lambda,z\rangle} dz\\
    & = \int_{\R^d} P_{t_1,\dots,t_n}(A-u)p(u+x)e^{-\langle
      \lambda,u+x\rangle} du\\
    & = \sum_{\beta\ge0} \frac{1}{\beta!}x^{\beta} e^{-\langle
      \lambda,x\rangle}\int_{\R^d} P_{t_1,\dots,t_n}(A-u)q_\beta(u)e^{-\langle
      \lambda,u\rangle} du\,,
  \end{align*}
  where $q_\beta$ is the partial derivative of $p$ of order $\beta$
  and $\beta!=\beta^1!\cdots\beta^d!$. The stationarity of the
  particle system and the uniqueness of the polynomial imply that the
  coefficients of the polynomial do not change, and so the statement
  of the theorem follows.
\end{proof}

\begin{theorem}
  \label{thr:sums-polynomials}
  The process $\xi$ and $\Lambda$ with density $\sum_{i=1}^n
  p_i(x)e^{-\langle \lambda_i,x\rangle}$, where $p_1,\dots,p_n$ are
  polynomials and $\lambda_1,\dots,\lambda_n\in \R^d$, generate a
  stationary particle system if and only if, for all $i=1,\dots,n$,
  the process $\xi$ and measure with density $p_i(x)e^{-\langle
    \lambda_i,x\rangle}$ form stationary particle systems.
\end{theorem}
\begin{proof}
  As in the proof of Theorem~\ref{lemma:mixture}, define $E_1$ and
  $E_2$, where now both these sets are finite. Since the exponential
  grows faster than polynomial, for each bounded Borel $A$ and
  sufficiently large $h$,
  \begin{align*}
    \Lambda_1*(P_{t_1,\dots,t_n}-P_{t_1+s,\dots,t_n+s})(A+hv) &>
    e^{-1}\Lambda_1*(P_{t_1,\dots,t_n}-P_{t_1+s,\dots,t_n+s})(A) \,,\\
    \Lambda_2*(P_{t_1+s,\dots,t_n+s}-P_{t_1,\dots,t_n})(A+hv) &\leq
    e^{-1}\Lambda_2*(P_{t_1+s,\dots,t_n+s}-P_{t_1,\dots,t_n})(A) \,,
  \end{align*}
  which eventually leads to a contradiction as in the proof of
  Theorem~\ref{lemma:mixture}.
\end{proof}

\begin{theorem}
  \label{thm:par-mu}
  Assume that $\E e^{\langle u, \xi(t) \rangle} < \infty$ for all
  $t\in\R$ and all $u$ from an open neighbourhood of $\lambda$. Then
  the process $\xi$ and the measure with exponential polynomial
  density $p(x)e^{-\langle\lambda,x\rangle}$ generate a stationary
  particle system if and only if
  \begin{multline}
    \label{eq:par-mu}
    q\Big(\frac{\partial}{\partial x}\Big)
    \varphi_{t_1,\dots,t_n}\Big(-\imath x
    -\sum_{i=2}^nu_i,u_2,\dots,u_n\Big)\Big\rvert_{x=\lambda}\\
    =q\Big(\frac{\partial}{\partial x}\Big)
    \varphi_{t_1+s,\dots,t_n+s}\Big(-\imath x
    -\sum_{i=2}^nu_i,u_2,\dots,u_n\Big)\Big\rvert_{x=\lambda} \,,
  \end{multline}
  for all partial derivatives $q$ of $p$, all $n\geq1$,
  $s,t_1,\dots,t_n\in\R$ and $u_1,\dots,u_n\in\R^d$ with
  $\sum_{i=1}^nu_i=0$.
\end{theorem}
\begin{proof}
  Denote shortly
  $\Delta\xi=(\xi(t_2)-\xi(t_1),\dots,\xi(t_n)-\xi(t_1))$. Similarly
  to (\ref{eq:mu-disint}), for a bounded Borel $A$,
  \begin{align*}
    \Lambda_{t_1,\dots,t_n} &(A) =
    \int_{R^d}\E \Big[ \mathbf{1}_{A-z}(0,\Delta \xi)
    e^{\langle\lambda,\xi(t_1)\rangle}p(z-\xi(t_1)) \Big]
    e^{-\langle\lambda,z \rangle}dz \nonumber \\
    &=\sum_{\beta \ge 0}\frac{1}{\beta!}(-1)^{\abs{\beta}}
    \int_{R^d}\E\Big[\mathbf{1}_{A-z}(0,\Delta\xi)\xi(t_1)^\beta
    e^{\langle x,\xi(t_1)\rangle}\Big]_{x=\lambda}q_\beta(z)
    e^{-\langle\lambda,z \rangle}dz \nonumber \\
    &= \sum_{\beta \ge 0}\frac{1}{\beta!}(-1)^{\abs{\beta}}
    \int_{R^d} \frac{\partial^{\abs{\beta}}}{\partial x^\beta}
    \mu_{t_1,\dots,t_n}(A-z)|_{x=\lambda} q_\beta(z)
    e^{-\langle\lambda,z \rangle}dz \,,
  \end{align*}
  where $q_\beta$ denotes the $\beta$'th partial derivative of $p$. By
  Theorems~\ref{thm:Laplace1e} and~\ref{thm:polynomial} the value of
  $\Lambda_{t_1,\dots,t_n}(A)$ is invariant for time shifts if and
  only if all the partial derivatives of $\mu$ are invariant. Taking
  the Fourier transform yields the claim.
\end{proof}

Now assume that $\xi$ is Gaussian. If $\xi$ and $\Lambda$ with density
$p(x)e^{-\langle\lambda,x\rangle}$ generate a stationary particle system,
then $\xi$ and $\efrak_\lambda$ also do, so that $\xi$ is described by
Theorem~\ref{thm:Gauss1e}. 

\begin{example}
  \label{ex:neg-poly}
  While in the univariate case Gaussian systems with a positive
  exponential polynomial density do not exist unless the polynomial
  part is constant, the convolution equation can be satisfied with a
  signed measure $\Lambda$.  For instance, one-dimensional signed
  measure on $\R$ with density $x^{2k+1}$ and the two-sided Brownian
  motion form stationary particle system for each $k\geq1$.
\end{example}

\subsection{Exponential measures on subspaces}
\label{sec:expon-meas-subsp}

Now assume that $\Lambda$ is supported by a linear subspace $\HH$ of
$\R^d$.  Denote by $\efrak_\lambda^\HH$ the measure on $\HH$ with
density $e^{-\langle \lambda,x\rangle}$, $x\in\HH$. The corresponding
Poisson point process is then a subset of $\HH$. Without loss of
generality, it is possible to assume that $\lambda\in\HH$ and
otherwise consider its orthogonal projection on $\HH$, which results
in the same density. Denote by $\xi^\HH(t)$ the orthogonal projection
of $\xi(t)$ onto $\HH$ and let $\xi^\perp=\xi-\xi^\HH$.

\begin{theorem}
  \label{thr:subspace}
  Assume that \eqref{eq:intexp} holds.  The process $\xi$ and
  $\efrak_\lambda^\HH$ generate a stationary particle system if and
  only if \eqref{eq:Fourier} holds for all $n\geq1$,
  $s,t_1,\dots,t_n\in\R$ and $u_1,\dots,u_n \in \R^d$ such that
  $\sum_{i=1}^n u_i$ is orthogonal to $\HH$.
\end{theorem}
\begin{proof}
  The proof is similar to the proof of Theorem~\ref{thm:Laplace1e}
  with 
  \begin{displaymath}
    \mu(A)=\E\big[\one_{(\xi(t_1)-\xi^\HH(t_1),\dots,
        \xi(t_n)-\xi^\HH(t_1))\in A}
      e^{\langle \lambda,\xi^\HH(t_1)\rangle}\big]\,.
  \end{displaymath}
\end{proof}

The following theorem concerns the Gaussian case. Let $m^\HH$,
$m^\perp$ be the expectations of $\xi^\HH$, $\xi^\perp$. Furthermore,
let $\Sigma^\HH(t_1,t_2)$ (respectively $\Sigma^\perp(t_1,t_2)$ and
$C(t_1,t_2)$) be the covariance matrix of $\xi^\HH(t_1)$ and
$\xi^\HH(t_2)$ (respectively of $\xi^\perp(t_1)$ and $\xi^\perp(t_2)$
and of $\xi^\HH(t_1)$ and $\xi^\perp(t_2)$). Finally,
$\Gamma^\HH(t_1,t_2)$ denotes the variogram of $\xi^\HH$.

\begin{theorem}
  \label{thr:subspace-gauss2}
  A Gaussian stochastic process $\xi$ and measure $\efrak_\lambda^\HH$ 
  generate a stationary particle system if and only if $\xi$ satisfies 
  the following conditions.
  \begin{itemize}
  \item[(i)] $\xi^\HH$ has representation \eqref{eq:Gauss1e} described
    in Theorem~\ref{thm:Gauss1e}.
  \item[(ii)] $\xi^\perp- m^\perp$ is stationary. 
  \item[(iii)]  $C(t_1,t_1)-C(t_2,t_1)=C(t_1+s,t_1+s)-C(t_2+s,t_1+s)$ \\ for all $s,t_1,t_2 \in \R$. 
  \item[(iv)] $m^\perp(t_2)+C(t_1,t_2)^\top\lambda=m^\perp(t_2+s)+C(t_1+s ,t_2+s)^\top\lambda$ \\ for all $s,t_1,t_2 \in \R$. 
  \end{itemize}
\end{theorem}
\begin{proof}
  By applying a linear transformation, it is easy to reduce the
  situation to the case of $\Lambda$ supported by the plane $\HH$
  spanned by the first $k<d$ basis vectors in $\R^d$. If
  $\xi(t)=(\xi^1(t),\dots,\xi^d(t))$, then $\xi^\HH=(\xi^1(t),\dots
  ,\xi^k(t),0,\dots,0)$ and
  $\xi^\perp=(0,\dots,0,\xi^{(k+1)}(t),\dots,\xi^d(t))$.  By
  Theorem~\ref{thr:subspace}, consider the Laplace transforms with
  $\lambda-\sum u_i$ being zeroes in its first $k$ coordinates.  As in
  Theorem~\ref{thm:Gauss1e}, consider the space $\LL$ that contains
  $(u_1,u_2)$ with $u_1^i+u_2^i=0$ for $i=1,\dots,k$ and
  $a=(\lambda,0)$. Then
  \begin{displaymath}
    A=
    \begin{pmatrix}
      I & -I_k \\
      -I_k & I
    \end{pmatrix}
  \end{displaymath}
  is the projection on $\LL$, where $I_k$ is the matrix with first $k$
  diagonal entries being one and otherwise zeroes. Then
  \begin{align}
    \label{eq:SigmaHH}
    A^\top\Sigma A &=
    \begin{pmatrix}
      \Gamma^\HH_{12}&C_{11}-C_{21}&-\Gamma^\HH_{12}&
      C_{12}-C_{22} \\
      C^\top_{11}-C^\top_{21}&\Sigma^\perp_{11}&C^\top_{21}
      -C^\top_{11} & \Sigma^\perp_{12}\\
      -\Gamma^\HH_{12}&C_{21}-C_{11}&\Gamma^\HH_{12}&
      C_{22}-C_{12} \\
      C^\top_{12}-C^\top_{22}&\Sigma^\perp_{21}&C^\top_{22}
      -C^\top_{12} & \Sigma^\perp_{22}\\
    \end{pmatrix}\,,\\[5pt]
    \label{eq:muHH}
    A^\top(m&+\Sigma a)=
    \begin{pmatrix}
      m^\HH_1+\Sigma^\HH_{11}\lambda-m^\HH_2-\Sigma^\HH_{21}\lambda\\
      m^\perp_1+C^\top_{11}\lambda \\
      m^\HH_2+\Sigma^\HH_{21}\lambda-m^\HH_1-\Sigma^\HH_{11}\lambda\\
      m^\perp_2+C^\top_{12}\lambda
    \end{pmatrix}\,,\\[4pt]
    \label{eq:constHH}
    \langle m,a\rangle&+\thf\langle a,\Sigma a\rangle
    =\langle m_1^\HH,\lambda \rangle+\thf \langle \lambda, \Sigma_{11}^\HH 
    \lambda \rangle\,,
  \end{align} 
  where $\Sigma_{ij}=\Sigma(t_i,t_j)$, $\Gamma_{ij}=\Gamma(t_i,t_j)$
  and $C_{ij}=C(t_i,t_j)$.  The invariance of $\Gamma_{ij}^\HH$,
  the first row of (\ref{eq:muHH}) and (\ref{eq:constHH}) imply the
  representation of $\xi^\HH$. The invariance of $\Sigma^\perp_{ij}$
  in (\ref{eq:SigmaHH}) yields the stationarity of
  $\xi^\perp-m^\perp$. The remaining entries of \eqref{eq:SigmaHH} are all of the form $\pm(C(t_1,t_1)-C(t_2,t_1))$ for $t_1,t_2\in \R$, which leads to condition (iii). 
Finally (iv) is obtained by considering the second and
fourth rows of (\ref{eq:muHH}).
  
  For the sufficiency note that the Laplace transform of the random
  vector $(\xi(t_1),\dots,\xi(t_n))$ at the point
  \begin{displaymath}
  \left(\begin{pmatrix} \lambda -\sum_{i=2}^nu_i^\HH\\u_1^\perp\end{pmatrix}, 
  \begin{pmatrix} u_2^\HH\\u_2^\perp  \end{pmatrix}, \dots,
  \begin{pmatrix} u_n^\HH\\u_n^\perp  \end{pmatrix} \right)
  \end{displaymath}
  consists of a combination of similar elements as given by
  (\ref{eq:SigmaHH}), (\ref{eq:muHH}) and (\ref{eq:constHH}), where
  $u_i^\HH$ and $u_i^\perp$ denote orthogonal projection of $u_i$ on
  $\HH$ and its orthogonal complement.
\end{proof}

The following example shows that there exist stationary systems
generated by a process $\xi$ with non-stationary $\xi^\perp$.

\begin{example}
  \label{ex:nonzeroB}
  Let $\HH=\R\times\{0\} \subset \R^2$ and consider
  $\xi=(\xi^1,\xi^2)$ with $\xi^1(t)=Zt-\thf\lambda t^2$ and
  $\xi^2(t)=Z-\lambda t$ for the standard Gaussian variable $Z$ and
  $\lambda\in\R$. By Theorem~\ref{thr:subspace-gauss2}, $\xi$ and
  $\efrak_{(\lambda,0)}^\HH$ form a stationary particle system.
\end{example}

\section{Multivariate Brown--Resnick processes}
\label{sec:MBR}

Consider a special case of the particle system that appears if the
Poisson process $\Pi$ lives on the diagonal line
$\HH=\{x^1=\cdots=x^d\}$ in $\R^d$. In this case, instead of the
additive particle system it is convenient to consider the
\emph{multiplicative} particle system
\begin{equation}
  \label{eq:Ne}
  N^e(t)=\{y_ie^{\xi_i(t)}, i \geq1\}\,,\qquad t\in\R\,,
\end{equation}
where $\{y_i:\; i\geq 1\}=\Pi^e$ is a Poisson process on $(0,\infty)$
with intensity measure $\Lambda^e$ and independent of i.i.d. copies
$\{\xi_n,\; n\geq1\}$ of an $\R^d$-valued stochastic process $\xi(t)$
satisfying
\begin{equation}
  \label{eq:integrability}
  \E e^{\xi(t)} < \infty \quad \text{for all } t \in \R \,.
\end{equation}
Note that the exponential is applied coordinatewisely and the
finiteness of expectation means that all its coordinates are
finite. Then the intensity measure of $N^e(t_1,\dots,t_n)$ is locally
finite and given by
\begin{displaymath}
  \Lambda^e_{t_1,\dots,t_n}(A)
  =\int_{(0,\infty)} 
  \Prob{(e^{\xi(t_1)},\dots,e^{\xi(t_n)}) \in y^{-1}A} \Lambda^e(dy)
\end{displaymath}
for all Borel $A \subset\R^{dn}$.

Assume that $\Pi^e$ has intensity measure $\Lambda^e(dy)=y^{-2}dy$,
$y>0$, and define a process $\eta$ with values in $\R^d$ by
\begin{equation}
  \label{eq:eta}
  \eta(t)=\bigvee_{i=1}^\infty y_ie^{\xi_i(t)}\,,  \quad t\in\R\,,
\end{equation}
where the maximum is taken coordinatewisely. It is well known that the
process $\eta$ is max-stable with unit Fr\'echet margins, see
\cite{fal:hus:rei04}. In order to determine the finite-dimensional
distributions of $\eta$ note that the event $\{\eta(t_1)\le
z_1,\dots,\eta(t_n) \le z_n\}$ (with coordinatewise inequalities) is
equivalent to the fact that no point of the process $N^e$ defined by
\eqref{eq:Ne} lies outside $A=(0,z_1]\times\cdots\times(0,z_n]$. The
latter probability equals
$\exp\{-\Lambda^e_{t_1,\dots,t_n}((0,\infty)^{dn}\setminus A)\}$, so
that 
\begin{equation}
  \label{eq:fidi}
  \Prob{\eta(t_1)\le z_1,\dots,\eta(t_n) \le z_n}
  =\exp\left\{-\E \max_{j,k}\left(e^{\xi^k(t_j)}/z_j^k\right)\right\}
\end{equation}
for all $t_1,\dots,t_n \in \R$ and $z_1,\dots,z_n \in \R^d$.
Applying this for $n=1$, it is easily seen that condition
(\ref{eq:integrability}) ensures that $\eta(t)$ is a.s. finite for
all $t\in \R$. Furthermore, the above argument shows that the
finite-dimensional distributions of $\eta$ uniquely determine the
finite-dimensional distributions of $N^e$.
In particular, $N^e$ is stationary if and only if $\eta$ is
stationary. The following definition appears in \cite{kab:sch:haan09},
however only for stochastic processes with values in the real line.

\begin{definition}[see \cite{kab:sch:haan09}]
  \label{def:BRstat}
  A stochastic process $\xi$ satisfying (\ref{eq:integrability}) is
  called \emph{Brown--Resnick stationary} if the process $\eta$
  defined by (\ref{eq:eta}) is stationary.
\end{definition}

\begin{theorem}
  \label{thm:BRLaplace}
  A stochastic process $\xi(t)$, $t \in \R$, is Brown--Resnick
  stationary if and only if
  \begin{displaymath}
    \varphi_{t_1,\dots,t_n}(u_1-\imath d^{-1}\mathbf{1},u_2,\dots,u_n)=
    \varphi_{t_1+s,\dots,t_n+s}(u_1-\imath d^{-1}\mathbf{1},u_2,\dots,u_n)
  \end{displaymath}
  for all $s,t_1,\dots,t_n \in \R$ and $u_1,\dots,u_n \in
  \R^d$ satisfying $\sum_{i,j=1}^n u_i^j=0$, where $\mathbf{1}$ is the 
  vector with all components equal to one.
\end{theorem}
\begin{proof}
  If we consider the additive system, i.e. let $x_i=\log(y_i)$ and
  $N(t)=\{x_i+\xi_i(t), i \geq1 \}$, then the Brown--Resnick
  construction corresponds to the situation where $\{x_i,i\geq1\}$ is
  a Poisson process on the line $\HH=\{(x,\dots,x):\; x\in\R\}$ in
  $\R^d$ with intensity $e^{-x}$, $x\in\R$.  Then
  the result follows from Theorem~\ref{thr:subspace}.
\end{proof} 

Since the measure $\Lambda^e$ is prescribed, the Brown--Resnick
stationary processes form a subclass of stationary particle systems
with special intensity measures supported by the
line $\HH$. 

In the following we characterise all pairs of a Gaussian process $\xi$
and a Poisson process on $\HH$ that yield stationary particle
systems. Their multiplicative variants may be regarded as
generalisations of Brown--Resnick stationary processes allowing for
general measures $\Lambda^e$. Note that $\xi^\HH$ is the vector with
all components being $\bar\xi =d^{-1}\sum_{i=1}^d\xi^i$ and
$\xi^\perp=\xi-\xi^\HH$.

\begin{theorem}
  \label{thr:gen-BR}
  A Gaussian process $\xi$, such that $\langle v,\xi\rangle$ is not
  stationary for some $v\notin\HH$, and a locally finite measure
  $\Lambda$ on the diagonal line $\HH$ in $\R^d$ generate a stationary
  particle system if and only if $\Lambda$ is proportional to
  $\efrak_\lambda^\HH$,
  \begin{displaymath}
    \bar{\xi}(t)=
    \begin{cases}
      W(t)+b(t)+c & \text{if }\; \lambda=0\,,\\
      W(t)-\thf \lambda \sigma^2(t)+c & \text{if }\;\lambda\neq0\,,
    \end{cases}    
  \end{displaymath}
  where $W(t)$, $t\in\R$, is a centred univariate Gaussian process
  with stationary increments and variance $\sigma^2(t)$, $b$ is an
  additive univariate function, $c\in\R$ is a constant, and $\xi^\HH, 
  \xi^\perp$ satisfy the conditions (ii)-(iv) of 
  Theorem~\ref{thr:subspace-gauss2}.
\end{theorem}
\begin{proof}
  The sufficiency is easy to show.  For the necessity, note that since
  $\Lambda$ is supported by $\HH$, the projected particle system
  $N_v=\{\langle v,x_i+\xi_i\rangle, i\ge 1\}$ is also a particle
  system generated by a non-stationary Gaussian process.  The
  characterisation of univariate particle systems from \cite{kab10}
  yields that $\Lambda$ is an exponential measure. The proof is
  completed by referring to Theorem~\ref{thr:subspace-gauss2}.
\end{proof}

\begin{example}
  Consider the two-dimensional process
  \begin{displaymath}
    \xi=\begin{pmatrix}
      W+\tilde{W}-\abs{t}/2\\
      W-\tilde{W}-\abs{t}/2
    \end{pmatrix},
  \end{displaymath}
  where $W$ is the one-dimensional two-sided Brownian motion and
  $\tilde{W}$ is any one-dimensional stationary Gaussian process
  independent of $W$. Then $\bar{\xi}=W-t/2$ and
  $\xi^\perp=(\tilde{W},-\tilde{W})^\top$ satisfy the conditions in
  Theorem~\ref{thr:gen-BR} with $\lambda=1$. In this case the measure
  $\Lambda^e$ on $(0,\infty)$ has density $y^{-2}$, $y>0$, so that the
  process $\xi$ is Brown--Resnick stationary.
\end{example}

\section*{Acknowledgements}

The authors are grateful to Zakhar Kabluchko for helpful discussions
at the earlier stage of this work. Special thanks goes to Marco Oesting who pointed out an error in an earlier version of Theorem~4.16. The comments by a referee have led
to several improvements in the presentation. This work is supported by
Swiss National Science Foundation Project Nr. 200021-137527.

\newcommand{\noopsort}[1]{} \newcommand{\printfirst}[2]{#1}
  \newcommand{\singleletter}[1]{#1} \newcommand{\switchargs}[2]{#2#1}

\end{document}